\documentclass[12pt,a4paper]{article} \usepackage[utf8x]{inputenc}
\usepackage{ucs} \usepackage[top=1in, bottom=1in, left=.75in,right=.75in]{geometry}
\usepackage{amsmath} 
\usepackage[nottoc, notlof,notlot]{tocbibind} 
\usepackage{amsfonts,dsfont} \usepackage{amssymb}
\usepackage{mathtools} 
\usepackage{amsthm} \usepackage{graphicx}
\usepackage{enumerate} 
\usepackage{color}
\author{Christophe Sabot, Pierre Tarr\`es, Xiaolin Zeng}
\title{The Vertex Reinforced Jump Process and a Random Schrödinger operator on finite graphs}
\newtheorem{lem}{Lemma}
\newtheorem{thm}{Theorem}
\newtheorem{defi}{Definition}
 \newtheorem{coro}{Corollary}
\newtheorem{prop}{Proposition}
\newtheorem{clai}{Claim} \usepackage{thmtools}
\declaretheoremstyle[notefont=\bfseries,notebraces={}{},%
headpunct={},postheadspace=1em]{mystyle}
\declaretheorem[style=mystyle,numbered=no,name=Theorem]{thm-hand}

\def\P{\mathbb{ P}} 
\def\R{\mathbb{ R}}
\def\ddd{{\mathcal D}}
\def\indic{{\mathds{1}}}
\def\qqq{{\mathcal Q}}
\def\demi{{1\over 2}}
\def\be{\beta}
\def\g{\gamma}
\def\G{\Gamma}
\newcommand{\Tc}{\cal T}
\newcommand{\s}{\sigma}
\def\NN{\mathbb{N}}

\newcommand{\tbe}{\tilde{\beta}}

\graphicspath{ {./images/}
} 
\date{}
\begin{document}
\maketitle
\begin{abstract}
  We introduce a new exponential family of probability distributions, which can be viewed as a multivariate
   generalization of the Inverse Gaussian distribution.
  Considered as the potential of a random Schr\"odinger operator, this exponential family is related to the random field that gives the mixing measure
  of the Vertex Reinforced Jump Process (VRJP), and hence to the mixing measure of the Edge Reinforced Random Walk (ERRW), the so-called magic formula.
  In particular, it yields by direct computation the value of the normalizing constants of these mixing measures, which solves a question raised by Diaconis.
The results of this paper are instrumental in \cite{Sabot-Zeng2015}, where several properties of the VRJP and the ERRW are proved, in particular
a functional central limit theorem in transient regimes, and recurrence of the 2-dimensional ERRW.
\end{abstract}
\section{Introduction}
\label{sec_intro}
In this paper we introduce a new multivariate exponential family, which is
a multivariate generalization of the inverse Gaussian law. This exponential family is associated
 to a network of conductances and provides a random field on the vertices of the network, the latter having
the remarkable property that the marginals have inverse gaussian law and that the field is decorrelated at distance
two.

This exponential family is mainly motivated by the study of two self-interacting processes, namely the Edge Reinforced Random Walk (ERRW) and
the closely related Vertex Reinforced Jump Process (VRJP), but it could also find some applications
in other topics, such as Bayesian statistics for instance. Note that Diaconis and Rolles \cite{diaconis2006bayesian} introduced in 2006 a family of Bayesian priors for reversible Markov chains, similarly associated to the limit measure of the ERRW. 

More precisely, we consider a non-directed finite graph \(\mathcal{G}=(V,E)\) with strictly positive conductances \(W_{i,j}=W_{j,i}\)
on the edges. Denote by
\(\Delta^W\) the discrete Laplace operator associated to the conductance network \((W_{i,j})\) and write \(W_i=\sum_{j: \{i,j\}\in E}W_{i,j}\).
The exponential family provides a random vector of positive reals \((\beta_j)_{j\in V}\)
such that
\[H_\beta:=-\Delta^W+V\]
is a.s.\ a positive operator, where \(V=2\beta-W\) is
the operator of multiplication by \((2\beta_i-W_i)\) and \(2\beta-W\) is considered as a random potential. We prove in Theorem~\ref{thm3} that if the Green function is defined by \(G=(H_\beta)^{-1}\), then the field \((e^{u_j})\) giving the mixing measure of the VRJP starting from \(i_0\), c.f.~\cite{sabot2011edge}, is equal in law to \((G(i_0,j)/G(i_0,i_0))\).

This has several consequences. Firstly, it relates the VRJP to a random Schr\"odinger operator with an explicit random potential with decorrelation at distance 2.
Note that Anderson localization was the main motivation in the papers of Disertori, Spencer, Zirnbauer (\cite{disertori2010anderson,disertori2010quasi}): in these works the supersymmetric
field related to the mixing measure of the VRJP (c.f.~\cite{sabot2011edge}) is viewed as a toy model for some supersymmetric fields
that appears in the physics literature in connection with random band matrices.
Secondly, it enables one to couple the mixing fields of the VRJP starting from different points. Finally, using the link between VRJP and ERRW \cite{sabot2011edge}, it yields an answer to
an old question of Diaconis about the direct computation of the normalizing constant of the `magic formula' for the mixing measure of ERRW.

Results of this paper are instrumental in \cite{Sabot-Zeng2015}, where the representation in terms of a random Schr\"odinger operator is extended to infinite graphs. Interesting
new phenomena appear in the transient case, where a generalized eigenfunction of the Schr\"odinger operator is involved in the representation.
Several consequences follow on the behavior of the VRJP and the ERRW in \cite{Sabot-Zeng2015}:
in particular a functional central limit theorem is proved for the VRJP
and the ERRW in dimension $d\ge 3$ at weak reinforcement, and recurrence of the 2-dimensional ERRW is shown, giving a full answer to an old question of
Diaconis.

The paper is organized as follows. In Section
\ref{sec_exp}, we define the new exponential family of distributions and give its first properties. 
In section
\ref{sec:vrjp}, we discuss the link between the exponential family and the Vertex reinforced jump
processes. In Section~\ref{sec_rw} we consider the ERRW and answer the question of Diaconis.
Sections~\ref{sec_proof_th1} and~\ref{s_proof_thm3} provide the proof of the two main results, namely Theorem~\ref{thm_main} and Theorem~\ref{thm3}.
\section{A new exponential family}
\label{sec_exp}
Let \(V=\{1, \ldots, n\}\) be a finite set, and let
\((W_{i,j})_{i\neq j}\) be a set of non-negative reals with \(W_{i,j}=W_{j,i}\ge 0\). Denote by \(E\) the edges associated to the positive
\(W_{i,j}\), i.e. consider the graph \({\mathcal G}= (V,E)\) with \(\{i,j\}\in E\) if and only if \(W_{i,j}>0\), and write \(i\sim j\) if \(\{i,j\}\in E\). Let \(\hbox{d}_{\mathcal G}\) be
the graph distance on \(\mathcal G\). 

When \(A\) is a symmetric operator on \(\mathbb{R}^{V}\) (also be considered as a \(V\times V\) matrix),  write \(A>0\) if \(A\) is positive definite, and \(|A|\) for its determinant.
\begin{thm}
\label{thm_main}
  Let \(P=(P_{i,j})_{1\leq i,j\leq n}\) be the symmetric matrix given by 
  \[
  P_{i,j}=
  \begin{cases}
    0 & i=j,\\
    W_{i,j} & i\neq j.
  \end{cases}
  \] 
  For any \(\theta\in \mathbb{R}_+^n\), we have
  \begin{align}
    \label{laplace}
 (\frac{2}{\pi})^{n/2}    \int \indic_{\{2\beta-P>0\}} e^{-\left<\theta,\beta\right>}\frac{d\beta}{\sqrt{|2\beta-P|}} =\exp\left(-\sum_{\{i,j\}\in E}W_{i,j}\sqrt{\theta_i\theta_j}\right)\cdot
    \prod_{i=1}^n\frac{1}{\sqrt{\theta_i}}
  \end{align}
  where \(d\beta=d\beta_1\cdots d\beta_n\), and \(2\beta -P\) is the operator on \(\mathbb{R}^{V}\) defined by
\[[(2\beta-P)f](i)=2\beta_{i}f(i)-\sum_{j:j\sim i}W_{i,j}f(j).\]
\end{thm}
\begin{defi}
\label{def_exp_fam}
The exponential family of random probability measures \(\nu^{W,\theta}(d\beta)\) is defined by
  \[  \nu^{W,\theta}(d\beta)= \mathds{1}_{2\beta-P>0}(\frac{2}{\pi})^{n/2}\exp\left(-\left<\theta,\beta\right>+\sum_{\{i,j\}\in E}W_{i,j}\sqrt{\theta_{i}\theta_{j}}\right) \frac{\prod_{i}\sqrt{\theta_{i}}}{\sqrt{|2\beta-P|}}d\beta\]
where \(\left<\theta,\beta\right>=\sum_{i\in V}\theta_{i}\beta_{i}\). We will simply write \(\nu^W\) for \(\nu^{W,1}\) in the case where \(\theta_i=1\) for all \(i\in V\).
\end{defi}
The proof of Theorem \ref{thm_main} is given in Section~\ref{sec_proof_th1}. We deduce the following simple but important properties of the measure \(\nu^{W,\theta}\).
\begin{prop}
\label{prop-main}
The Laplace transform of \(\nu^{W,\theta}\) is
\[
\int e^{-\left<\lambda,\beta\right>}   \nu^{W,\theta}(d\beta)=
\exp\left(-\sum_{\{i,j\}\in E}W_{i,j}\left(\sqrt{\lambda_i+\theta_i}\sqrt{\lambda_j+\theta_j}-\sqrt{\theta_i\theta_j}\right)\right)\cdot
    \prod_{i=1}^n\sqrt{\frac{\theta_{i}}{{\lambda_i+\theta_i}}} 
\]
Moreover, if \(\beta\) is a random vector with distribution \( \nu^{W,\theta}\), then
  \begin{itemize}
\item
The marginals \(\beta_i\) are such that \(\frac{1}{2\beta_{i}\theta_{i}}\) is an Inverse Gaussian
  distribution with parameters \- \((\frac{1}{\sum_{j\sim i}W_{i,j}\sqrt{\theta_{i}\theta_{j}}},1)\)
  \item
If \(V_1\subset V\), \(V_2\subset V\) are two subsets of \(V\) such that \(d_\mathcal{G}(V_1,V_2)\ge 2\), then
\((\beta_i)_{i\in V_1}\) and \((\beta_j)_{j\in V_2}\) are independent.
\end{itemize}
\end{prop}
\begin{proof} 
The Laplace transform of \(\nu^{W,\theta}\) can be computed directly from Theorem~\ref{thm_main}, from which we deduce independence at distance at least \(2\). We can also deduce, by identification of the Laplace transforms, that the marginals of this law are reciprocal inverse gaussian up to a multiplicative constant.
\end{proof}
The family can be reduced to the case \(\theta=1\) by changing \(W\), as shown in the next corollary.
\begin{coro}
Let \((\beta_j)_{j\in V}\) be distributed according to \(\nu^{W,\theta}\). Then \((\theta\beta)\) is distributed
according to \(\nu^{W^\theta}\), where \(W^\theta_{i,j}=W_{i,j}\sqrt{\theta_i\theta_j}\).
\end{coro}

It is clear from the expression of the Laplace transform that if the graph has several connected components then
the random field \((\beta_j)_{j\in V}\) splits accordingly into independent random subvectors. Therefore, we will always assume in
the sequel that the graph \(\mathcal{G}\) is connected. 

\section{Link with the Vertex reinforced Jump process}
\label{sec:vrjp}
\subsection{Vertex Reinforced Jump Process: definition and main properties}
In this section we explain the link between the exponential family of Section~\ref{sec_exp}
and the Vertex reinforced Jump Process (VRJP), which is a linearly
reinforced process in continuous time, defined in~\cite{davis2002continuous}, investigated on trees in \cite{basdevant2012continuous}, 
and on general graphs by the first two authors in~\cite{sabot2011edge}. 
Consider as in the previous section a conductance network \((W_{i,j})\)
and the associated graph \(\mathcal{G}=(V,E)\). Fix also some positive parameters \((\phi_i)_{i\in V}\) on the vertices. 
Assume that the graph \(\mathcal{G}\) is connected.

We call VRJP with conductances \((W_{i,j})\) and initial local time \((\phi_i)\)
the continuous-time process \((Y_t)_{t\ge0}\) on \(V\), starting at time \(0\) at some vertex \(i_0\in V\) and such that, 
if \(Y\) is at a vertex \(i\in V\) at time \(t\), then, conditionally on \((Y_s, s\le t)\), the process jumps to a neighbour \(j\) of \(i\) at rate 
\[W_{i,j}L_j(t),\]
where
\[L_j(t):=\phi_j+\int_0^t \mathds{1}_{\{Y_s=j\}}\,ds.\]
The following time change, introduced in~\cite{sabot2011edge}, plays a central role. Let 
\begin{equation}
\label{Dt}
D(t)=\sum_{i\in V}{(L_i^2(t)-\phi^2_i)},
\end{equation}
 define
\(Z_t\) as the time changed process
\[Z_t=Y_{D^{-1}(t)}.
\]
Let \((\ell_j(t))\) be the local time of \(Z\) at time \(t\) (that is, \(\ell_{j}(t)=\int_{0}^{t}\mathds{1}_{Z_{s}=j}ds\)).
Conditionally on the past, at time \(t\), the process \(Z\) jumps from \(Z_t=i\) to a neighbour \(j\) at rate (c.f.~\cite{sabot2013ray}, Lemma 3)
\[
\frac{W_{i,j}}{2}\sqrt{\frac{\phi^2_j+\ell_j(t)}{\phi^2_i+\ell_i(t)}}.
\]
We state below one of the main results of~\cite{sabot2011edge}, Proposition 1 and Theorem 2. The theorem was stated in~\cite{sabot2011edge}
in the case \(\phi=1\), this version of the theorem can be deduced by a simple change of time, details are given in Appendix~\ref{Appendix}.
\begin{thm}
\label{thm-ST}
Assume that \(\mathcal{G}\) is finite. Suppose that the VRJP starts at \(i_0\). The limit
\[
U_{i}=\frac{1}{2}\lim_{t\to\infty}\left(\log\left(\frac{\ell_{i}(t)+\phi_{i}^{2}}{\ell_{i_{0}}(t)+\phi_{i_{0}}^{2}}\right)-\log\left(\frac{\phi^2_{i}}{\phi^2_{i_0}}\right)\right)
\]
exists a.s.\ and, conditionally on \(U\), \(Z\) is a Markov jump processes with jump rate from \(i\) to \(j\)
\[
{1 \over 2} W_{i,j}e^{U_j-U_i}.
\]
Moreover \((U_j)\) has the following distribution on \(\{(u_i), \;\; u_{i_0}=0\}\)
\begin{equation}
\label{density_u}
\mathcal{Q}^{W,\phi}_{i_0}(du)={\prod_{j\neq i_0} \phi_j \over \sqrt{2\pi}^{\vert V\vert-1}} e^{-\sum_{j\in V} u_j} e^{-\frac{1}{2}\sum_{\{i,j\}\in E}W_{i,j}(e^{u_i-u_j}\phi_j^2+e^{u_j-u_i}\phi_i^2-2\phi_i\phi_j)
        } \sqrt{D(W,u)}\; du,
\end{equation}
with \(du = \prod_{j\in V\setminus\{i_0\}} du_j\) and
\[D(W,u)=\sum_T \prod_{\{i,j\} \in T} W_{i,j} e^{u_i+u_j}
\] 
where the sum runs on the set of spanning trees \(T\) of \(\mathcal{G}\). We simply write \(\mathcal{Q}_{i_{0}}^{W}\) for \(\mathcal{Q}_{i_{0}}^{W,1}\)
\end{thm}
The fact that the total mass of the measure \(\mathcal{Q}_{i_0}^{W,\phi}\) is 1 is both a non-trivial and a useful fact: in particular, it plays a central role in the delocalization and
localization results of~\cite{disertori2010anderson,disertori2010quasi}. In~\cite{sabot2011edge} it is a consequence of the fact that it is the probability distribution of the limit random variables \(U\).
In~\cite{disertori2010quasi} it is proved using a sophisticated supersymmetric argument, the so-called localization principle. Theorem~\ref{thm3} below provides a direct 'computational' proof of that result,  based on the identity (\ref{laplace}) and on the change of variable in Proposition \ref{p:green} that relates the field \((u_j)\) to the random vector \((\beta_j)\) in Definition \ref{def_exp_fam}.
\subsection{Link with the random potential \(\beta\)}
The second main result of this paper enables us to construct the mixing field \(e^u\) defined  in the previous subsection from the random potential
\((\beta_j)\) defined in Definition~\ref{def_exp_fam}. It gives also a natural way to couple the mixing measure of VRJP starting from different points.

Let us first state the following Proposition \ref{p:green}, which provides some elementary observations on the Green function.

Define
\[
\ddd=\{ (\beta_i)_{i\in V} \in (\R_+\setminus\{0\})^V,\; \; 2\beta -P>0\}.
\]

\begin{prop}
\label{p:green}
Let \(\be\in\ddd\), and let \(G\) be the inverse of \((2\beta-P)\). Then \((G(i,j))\) has positive coefficients. Define \((u(i,j))_{i,j\in V}\) by
\[
e^{u(i,j)}={G(i,j)\over G(i,i)}.
\]
Then for \(i_0\in V\), the function \(j\rightarrow u(i_0,j)\) is the unique solution \(j\mapsto u_j\) of the equation
\begin{equation}
\label{system}
\begin{cases}
\sum_{j\sim i}\frac{1}{2} W_{i,j}e^{u_j-u_i}=\beta_i, & i\neq i_0
\\
u_{i_{0}}=0, & 
\end{cases}
\end{equation}
In particular \((u(i_0,j))_{j\in V}\) is \((\beta_j)_{j\in V\setminus\{i_0\}}\) measurable. Moreover, at the site \(i_0\) we have 
\[
\beta_{i_0}= {1\over 2G(i_0,i_0)} +\sum_{j: j\sim i_{0}} \frac{1}{2}W_{i_0,j}e^{u(i_{0},j)}.
\]
\end{prop}

\begin{thm}\label{thm3}
Let \(\beta\) be a random potential with distribution \(\nu^{W,\phi^{2}} (d\beta)\) as in Definition \ref{def_exp_fam}, and let \((u(i,j))_{i,j\in V}\) be defined as in Proposition \ref{p:green}. Then the following properties hold:
\begin{enumerate}[i)]
\item The random field \((u(i_0,j))_{j\in V}\) has the distribution of the mixing measure \(\mathcal{Q}^{W,\phi}_{i_0}(du)\) of the VRJP starting from \(i_0\) with initial local time \((\phi_i)_{i\in V}\).
\item  The random variable \(G(i_0,i_0)\) has the distribution of  \(1/(2\gamma)\), where \(\gamma\) is a gamma random variable with parameters \((1/2,1/\phi_{i_0}^2)\). Moreover,
\(G(i_0,i_0)\) is independent of \((\beta_j)_{j\neq i_0}\), and thus also of the field \((u(i_0,j))_{j\in V}\). \end{enumerate}
\end{thm}
The proofs of Proposition \ref{p:green} and Theorem \ref{thm3} are given in Section \ref{s_proof_thm3}. The next Corollary \ref{coro-thm3} describes how to construct the random potential \(\beta\) from the field \(u\) of Theorem~\ref{thm-ST}.
\begin{coro}
\label{coro-thm3}
  Consider a VRJP with edge weight \((W_{i,j})\) and initial local time \((\phi_i)_{i\in V}\), 
starting at \(i_0\). Let \((u_i)_{i\in V}\) be distributed according to \(\mathcal{Q}_{i_0}^{W,\phi}\) of
  Theorem~\ref{thm-ST}. Let
  \begin{equation}
  \label{def:tbe}
 \tilde{\beta}_i=\frac{1}{2}\sum_{j:j\sim
    i}W_{i,j}e^{u_j-u_i}.
\end{equation}

Let \(\gamma\) be a Gamma distributed random
  variable with parameters \((\frac{1}{2},1/\phi_{i_0}^2)\), independent of \((u_j)\), and let
  \begin{equation}
  \label{nbe}
  \beta_i=\tilde{\beta}_i+\mathds{1}_{i_0}\gamma.
  \end{equation}
  Then \(\beta\) has the law \(\nu^{W,\phi^{2}}\) of Definition \ref{def_exp_fam}.
\end{coro}
Corollary \ref{coro-thm3} indeed follows directly from Theorem \ref{thm3} and Proposition \ref{p:green}: the law of \(\be\) in \eqref{nbe} is uniquely determined by the laws of \((u_i)_{i\in V}\) and \(\g\) independent from \((\be_i)_{i\ne i_0}\), hence it is sufficient to show that, if \(\be\) has distribution \(\nu^{W,\phi^{2}} (d\beta)\) and \(u\) is defined from \eqref{system} by Proposition \ref{p:green}, then \((u_i)_{i\in V}\) indeed has distribution \(\mathcal{Q}_{i_0}^{W,\phi}\), and \(\g=\beta_{i_0}-\tilde{\beta}_{i_0}=1/(2G(i_0,i_0))\) has distribution \(\G(1/2,1/\phi_{i_0}^2)\), which follows from Theorem \ref{thm3}.

As mentioned in the introduction, Theorem~\ref{thm3} has several consequences.
Firstly it explicitly relates the VRJP to the random Schr\"odinger operator \(-\Delta^W+V\), where \(V\) is the random potential \(V_i=2\beta_i-W_i\). Secondly it yields a natural coupling between the random fields \((u_j)_{j\in V}\) associated with the VRJP starting from different sites, since the exponential family \((\beta_i)_{i\in V}\) gives the same role to each vertex of the graph, and 
\((u(i,j))_{i,j\in V}\) arises from these random variables \((\beta_i)_{i\in V}\).
Finally it also gives a computational proof of the identity \(\int \mathcal{Q}^{W,\theta}_{i_0}(du )=1\), for any \(\theta\), as a consequence of
Theorem~\ref{thm_main} that allows to define \(\nu^{W,\phi^{2}} (d\beta)\) as a probability measure.
\section{Link with the Edge reinforced random walk and a question of Diaconis}
\label{sec_rw}

\subsection{Definition and magic formula}

The Edge Reinforced Random Walk (ERRW) is a famous discrete time process introduced in 1986 by Coppersmith and Diaconis,~\cite{coppersmith1987random}.

Let \((a_{i,j})_{\{i,j\}\in E} \) be a set of positive weights on the edges of the graph \(\mathcal G\).
Let \((X_n)_{n\in\mathbb{N}}\) be a random process that takes values in \(V\), and let 
\(\mathcal{F}_n=\sigma(X_0,\ldots,X_n)\) be the filtration of its past. For any \(e\in E\), \(n\in\mathbb{N}\), let 
\begin{equation}\label{Zn}
Z_n(e)=a_e+ \sum_{k=1}^n\mathds{1}_{\{\{X_{k-1},X_k\}=e\}}
\end{equation}
be the number of crosses of the (non-directed) edge \(e\) up to time \(n\) plus the initial weight \(a_e\).

Then \((X_n)_{n\in\mathbb{N}}\) is called Edge Reinforced Random Walk (ERRW) with starting point \(i_0\in V\) and weights \((a_e)_{e\in E}\),  
if  \(X_0=i_0\) and,  for all \(n\in\mathbb{N}\), 
\begin{equation}
\label{def-errw}
\mathbb{P}(X_{n+1}=j~|~\mathcal{F}_n)=\mathds{1}_{\{j\sim X_n\}}\frac{Z_n(\{X_n,j\})}
{\sum_{k\sim X_n} Z_n(\{X_n,k\})}.
\end{equation}
We denote by \(\mathbb{P}^{ERRW,(a)}_{i_0}\) the law of the ERRW starting from the initial vertex \(i_0\) and initial weights \((a)\).

A fundamental property of the ERRW, stated in the next theorem, 
is that on finite graphs the ERRW is a mixture of reversible Markov chains, and the mixing measure can be determined explicitly 
(the so-called Coppersmith-Diaconis measure, or `magic formula'). It is a consequence of a de Finetti theorem for Markov chains due
to Diaconis and Freedman~\cite{diaconis1980finetti}, and the explicit determination of the law
is due to Diaconis and Coppersmith \cite{coppersmith1987random, rolles2000edge,merkl2008magic}. It has also applications in Bayesian statistics~\cite{bacallado2011bayesian,bacallado2009bayesian,diaconis2006bayesian}.

\begin{thm}{\cite{coppersmith1987random, rolles2000edge}}\label{thm3_cd}

Assume that \(\mathcal{G}=(V,E)\) is a finite graph and  set \(a_i=\sum_{j: \{i,j\} \in E}a_{i,j}\) for all \(i\in V\). Fix an edge
\(e_0\) incident to \(i_0\), and define \(\mathcal{H}_{e_0}=\{y: \,\forall e\in E,\ y_e>0,\
y_{e_0}=1\}\) (similarly let \(y_i=\sum_{i\in
  e}y_e\)).
Consider the following positive measure
defined on \(\mathcal{H}_{e_0}\) defined 
  by its density
\begin{align}
  \label{magic_formula}
  \mathcal{M}^{(a)}_{i_0}(dy)=C(a,i_0) \frac{\sqrt{y_{i_0}}\prod_{e\in
      E}y_e^{a_e}}{\prod_{i\in
      V}y_i^{\frac{1}{2}(a_i+1)}}\sqrt{D(y)}\prod_{e\neq e_0}\frac{d
    y_e}{y_e},
\end{align}
with 
\[D(y)=\sum_T \prod_{e \in T} y_e,
\] 
where the sum runs on the set of spanning trees \(T\) of \(\mathcal{G}\), and with
\[
C(a,i_0)= \frac{2^{1-|V|+\sum_{e\in
        E}a_e}}{\sqrt{\pi}^{|V|-1}}\cdot \frac{\prod_{i\in
      V}\Gamma(\frac{1}{2}(a_i+1-\mathds{1}_{i=i_0}))}{\prod_{e\in
      E}\Gamma(a_e)}
\]
Then \(\mathcal{M}^{(a)}_{i_0}\) is a probability measure on \(\mathcal{H}_{e_0}\), and it is the mixing measure of the ERRW
starting from \(i_0\), more precisely
\[
\P^{ERRW,(a)}_{i_0}(\cdot)=\int_{\mathcal{H}_{i_0}} P^{(y)}_{i_0}(\cdot)d
\mathcal{M}^{(a)}_{i_0}(y),
\]
where \(P^{(y)}_{i_0}\) denotes the reversible Markov chain starting at
\(i_0\) with conductances \((y)\).
\end{thm}

\subsection{The question of Diaconis}
The fact that \({\mathcal{M}}^{(a)}_{i_0}(dy)\) is a probability measure is a consequence of the fact that
it is the mixing measure of the ERRW.\@ In fact it is obtained as the limit distribution of the normalized occupation time of the edges \cite{rolles2000edge}:
\[
\left({Z_n(e)\over Z_n(e_0)}\right)_{e\in E}\overset{law}{\longrightarrow} \mathcal{M}^{(a)}_{i_0}.
\]
One question raised by Diaconis is the following

\begin{equation}
\label{q-diaconis}
\hbox{(Q)  Prove by direct computation that \(\int\mathcal{M}^{(a)}_{i_0} (dy) =1\).}
\end{equation}
An answer was proposed by Diaconis and Stong \cite{Diaconis-Stong} in the case of the triangle, using a subtle change of variables.
Also note that Merkl and Rolles offered in \cite{merkl2008magic}  analytic tools for the computation of the ratio of the normalizing constants of the magic formula for two initial weights differing by integer values, which may possibly be extended to provide the normalizing constant.

We provide below an answer to that question. A first simplification comes from~\cite{sabot2011edge}, where an explicit link was made between the VRJP and the ERRW.
\begin{thm}[Theorem 1, \cite{sabot2011edge}] 
Consider \((Y_n)\) the discrete time process associated with the VRJP \((Y_t)\) (i.e.\ taken at jump times) with conductances \((W_{i,j})\) and \(\phi=1\).
Take now the conductances \((W_e)_{e\in E}\) as independent random variables with 
gamma distribution with parameters \((a_e)_{e\in E}\). Then the `annealed' law of \(Y_n\) (i.e.\ the law after taking expectation with respect to 
the random \((W_e)\)) is the law of the ERRW \((X_n)\) with initial weights \((a_e)_{e\in E}\).
\end{thm}

This immediately implies an identity between the mixing measures \(\mathcal{M}^{(a)}_{i_0}\) and \(\mathcal{Q}^{W}_{i_0}\): indeed, by
Theorem~\ref{thm-ST}, \((Y_n)\) is a mixture of Markov jump processes with conductances \(W_{i,j}e^{u_i+u_j}\), which implies that
for all 0-homogeneous bounded test functions \(\phi\) (i.e. \(\phi(\lambda y)=\phi(y)\), \(\forall \lambda>0\)), we have
\begin{equation}
\label{eq3}
\int_{\mathcal{H}_{e_0}} \phi((y_e)) \mathcal{M}^{(a)}_{i_0}(dy)=  \int_{\mathbb{R}^E} \prod_{e\in E} {W_e^{a_e-1}e^{-W_e}\over
\Gamma(a_e)} \left(\int\phi((W_{i,j}e^{u_i+u_j})) \mathcal{Q}^{W}_{i_0}(du)\right)dW.
\end{equation}
with \(dW=\prod_{e\in E} dW_e\). This identity was checked by direct computation in section 5 of~\cite{sabot2011edge}.

Now, the fact that \(\int \mathcal{Q}_{i_0}^{W}(du)=1\) is a consequence of the computation of the integral~(\ref{laplace}) in Theorem~\ref{thm_main}
and the change of variables in Theorem~\ref{thm3}, as explained at the end of Section \ref{sec:vrjp}. Therefore
  \[\int_{y_{e_0}=1} d \mathcal{M}_{i_0}^a(y)=1.\]
  
Note that this fact can be used to prove directly that \({\mathcal M}_{i_0}^a(dy)\) is the mixing measure of the ERRW starting from
initial condition \((a)\) and initial vertex \(i_0\). Indeed, for any finite path \(\sigma:i_0\to i_1\to \cdots\to i_n\), let \(N(i)\) (resp. \(N(e)\)) be the number of times vertex \(i\) (resp. edge \(e\)) is visited (resp. crossed):
\begin{align*}
N(i)&=\vert\{k;\ 0\leq k \leq n-1, \; i_k=i\}\vert\\
N(e)&=\vert \{ k;\ 0\leq k\leq n-1,\; \{i_k,i_{k+1}\}=e\}\vert.
\end{align*}
The probability of \(\sigma\) for the reversible Markov
  chain of conductance \(y\) is
  \[p_{i_0}^y(\sigma)=\frac{\prod_{e\in E}y_e^{N(e)}}{\prod_{i\in
      V}y_i^{N(i)}}\] 
 The integration of \(p_{i_0}^y(\sigma)\) w.r.t.\ \(d \mathcal{M}_{i_0}^a(y)\) can be computed by changing the constant \(\Gamma(a_e)\) 
      to \(\Gamma(a_e+N_e)\) and \(\Gamma(\frac{1}{2}(a_i+1))\) to \(\Gamma(\frac{1}{2}(a_i+1)+N_i)\). Using the property \(\Gamma(x+1)=x\Gamma(x)\) and 
  the notation \((a,n)=\prod_{k=0}^{n-1} (a+k)\), we deduce
  \[\int p_{i_0}^y(\sigma) d \mathcal{M}_{i_0}^a(y)=\frac{\prod_e
    (a_e, N(e))}{\prod_i
    (a_i,N(i))}\] which is the
  probability of an ERRW to follow the path \(\sigma\).
\section{Proof of Theorem~\ref{thm_main}}
\label{sec_proof_th1}

\begin{lem}
\label{lem1}
  Let \(P=(P_{i,j})_{1\leq i,j\leq n}\) be a symmetric matrix with 
  \[P_{i,j}=
  \begin{cases}
    0, & i=j,\\
    W_{i,j},\in \mathbb{R}^+ & i\neq j,
  \end{cases}\]
 and let  \(\beta\) be a diagonal matrix with entries
   \(\beta_i,\ i=1,\ldots,n \), such that \(M=2\beta-P\) is positive
  definite. 
  
 Let \(L\) be the lower triangular \(n\times n\) matrix and \(U\) be the upper unitary (with \(1\) on the diagonal)  upper triangular matrix such that \(M=LU\) (i.e. the LU decomposition of \(M\)), which exist and are unique.
 
 Then 
 \[U=
  \begin{pmatrix}
    x_1& -H_{1,2}&\cdots&-H_{1,n}\\
    0&x_2&\cdots&-H_{2,n}\\
    &\cdots&&-H_{n-1,n}\\
    0&\cdots&0&x_n
  \end{pmatrix},
  \]
where \((x_i)_{1\le i\le n}\) and \((H_{i,j})_{1\le i<j\le n}\) are defined recursively by
\[ \begin{cases}H_{1,j}=W_{1,j} & j>1 \\
        H_{i,j}=W_{i,j}+\sum_{k=1}^{i-1}\frac{H_{k,i}H_{k,j}}{x_k} & i\geq 2, \ j>i\\
        x_{i}=2\beta_{i}-\sum_{k=1}^{i-1}\frac{H_{k,i}^{2}}{x_{k}}& i\ge 1.
  \end{cases}\]

  Furthermore, 
\[x_{i}=\frac{M(1,\ldots,i|1,\ldots,i)}{M(1,\ldots,i-1|1,\ldots,i-1)}\]
where \(M(I|J)\) is the minor of matrix \(M\) that corresponds to the rows with index in \(I\) and columns with index in \(J\). 
\end{lem}
  The result follows directly from (2.6) of~\cite{varga1981thelu}, but we prove it in Appendix~\ref{pfOfLemma1} for completeness' sake. 
\begin{clai}
\label{claim-laplace}
For any \(\theta_1>0\), \(\theta_2{ \ge}0\),
  \begin{align*}
  \int_0^{\infty}\exp(-\frac{\theta_1 x}{2}-\frac{\theta_{2}}{2x})\frac{1}{\sqrt{x}}dx=\exp(-\sqrt{\theta_1\theta_2})\sqrt{\frac{2\pi}{\theta_1}}.
  \end{align*}
\end{clai}
\begin{proof}
 The case \(\theta_2=0\) corresponds to the normalisation of the \(\Gamma(\demi)\) variable. The case \(\theta_2>0\) corresponds to the normalization of the Inverse Gaussian law \(\hbox{IG}({\theta_1\over \theta_2}, {1\over \theta_2})\).
\end{proof}
Let us now prove  Theorem~\ref{thm_main}.
In the sequel we take the convention, given any real sequence \((a_k)_{k\in\NN}\), that \(\sum_{k=n}^m a_k=0\) if \(n>m\).

By Lemma~\ref{lem1},
\[\sum_{k=1}^n \theta_l\beta_l=\sum_{k=1}^n\theta_k(\frac{x_k}{2}+\sum_{l=1}^{k-1}\frac{H_{l,k}^2}{2x_l})
  =\sum_{l=1}^n\left[\frac{\theta_l x_l}{2}+\frac{1}{2x_l}(\sum_{k=l+1}^n \theta_k H_{l,k}^2)\right].\]
  
Define
\begin{align*}
\Psi: (\R_+\setminus\{0\})^n &\longrightarrow \ddd\\
(x_{i})_{1\leq i\leq n}&\longmapsto (\beta_{i})_{1\leq i\leq n}=\left(\frac{x_{i}}{2}+\sum_{k=1}^{i-1}\frac{H_{k,i}^{2}}{x_{k}}\right)_{1\leq i\leq n}.
\end{align*}

Then \(\Psi\) is a bijection, since a symmetric matrix is positive definite if and only if all of its diagonal minors are positive. Its Jacobian is \(2^{-n}\), hence it is a diffeomorphism.

Therefore
\[I:= \int\mathds{1}_{\{2\beta -P>0\}}\frac{\exp(-\theta\beta)}{\sqrt{|2\beta -P|}}d\beta=\int_{\mathbb{R}_+^n}\exp\left(-\sum_{l=1}^n\left[\frac{\theta_l x_l}{2}+\frac{1}{2x_l}(\sum_{k=l+1}^n \theta_k H_{l,k}^2)\right]\right)\frac{1}{\sqrt{x_1\cdots x_n}}\frac{1}{2^n}dx.\]

Let, for all \(1\le l\le m\le n\),
\begin{align*}
R_{l,m}&=\left(\sum_{j=m+1}^{n}H_{l,j}\sqrt{\theta_{j}}\right)^{2}+\sum_{k=l+1}^{m}\theta_{k}H_{l,k}^{2}\\
S_{l,m}&=\frac{\theta_lx_l}{2}+\frac{R_{l,m}}{2x_l}.
\end{align*}
Note that \(R_{l,m}\) (resp. \(S_{l,m}\)) only depends on \(x_1\), \(\ldots\) \(x_{l-1}\) (resp. \(x_1\), \(\ldots\) \(x_{l}\)). 

Let, for all \(1\le m\le n\), 
\[I_m:=\int_{\mathbb{R}_+^m}\exp\left(-\sum_{l=1}^mS_{l,m}\right)\frac{dx_{1}\cdots dx_{m}}{\sqrt{x_1\cdots x_m}}.\]
We will take the convention that, if \(m=0\), the integral of \(dx_{1}\cdots dx_{m}\) is \(1\), so that \(I_0=1\).

Note that \(I=I_n/2^n\). We also have the following lemma. 
\begin{lem}
\label{lem:in}
For all \(1\le m\le n\), we have
\[I_m=\sqrt{\frac{2\pi}{\theta_m}}\exp\left(-\sum_{j=m+1}^nW_{m,j}\sqrt{\theta_m\theta_j}\right)I_{m-1}.\]
\end{lem}
\begin{proof}
Using {Claim} \ref{claim-laplace}, we deduce
\begin{align}
\nonumber
I_m&=\int_{\mathbb{R}_+^m}\exp\left(-\left[\frac{\theta_mx_m}{2}+\frac{R_{m,m}}{2x_m}+\sum_{l=1}^{m-1}S_{l,m}\right]\right)\frac{dx_{1}\cdots dx_{m}}{\sqrt{x_1\cdots x_m}}\\
\label{in}
&=\int_{\mathbb{R}_+^{m-1}}\exp\left(-\sqrt{R_{m,m}\theta_m}-\sum_{l=1}^{m-1}S_{l,m}
\right)\frac{dx_{1}\cdots dx_{m-1}}{\sqrt{x_1\cdots x_{m-1}}}.
\end{align}
Now 
\(R_{m,m}=\left(\sum_{j=m+1}^{n}H_{l,j}\sqrt{\theta_{j}}\right)^{2}\) and 
\[H_{m,j}=W_{m,j}+\sum_{l=1}^{m-1}\frac{H_{l,m}H_{l,j}}{x_l},\]
so that
\[\sqrt{R_{m,m}\theta_m}=\sum_{j=m+1}^nW_{m,j}\sqrt{\theta_m\theta_j}
+\sum_{l=1}^{m-1}\frac{H_{l,m}\sqrt{\theta_m}}{x_l}\sum_{j=m+1}^nH_{l,j}\sqrt{\theta_j}.\]
On the other hand, for all \(1\le l\le m-1\),
\[S_{l,m}-S_{l,m-1}=-\frac{H_{l,m}\sqrt{\theta_m}}{x_l}\sum_{j=m+1}^nH_{l,j}\sqrt{\theta_j}.\]
Therefore
\[\sqrt{R_{m,m}\theta_m}+\sum_{l=1}^{m-1}S_{l,m}
=\sum_{j=m+1}^nW_{m,j}\sqrt{\theta_m\theta_j}+\sum_{l=1}^{m-1}S_{l,m-1},\]
which enables to conclude by \eqref{in}.
\end{proof}
We deduce from Lemma \ref{lem:in}, by induction, that 
\[I=\frac{I_n}{2^n}
=\frac{1}{2^{n}}\sqrt{\frac{(2\pi)^{n}}{\theta_{n}\cdots \theta_{1}}}\exp\left(-\sum_{\{i,j\}\in E}W_{i,j}\sqrt{\theta_{i}\theta_{j}}\right),\]
which enables us to conclude.
\section{Proof of Proposition \ref{p:green} and Theorem \ref{thm3}}
\label{s_proof_thm3}
\subsection{Proof of Proposition \ref{p:green}}
Fix \(i_0\in V\), and let \(\be\in\ddd\).  Let us first justify the existence and uniqueness of \(u(i_{0},i)\) defined by the linear
  system~\eqref{system}. 
  As \((2\beta-P)\) is an M-matrix, its inverse \(G\) satisfies \(G(i,j)>0\) for any \(i,j\).
A solution \((u_j)\) of equation~\eqref{system} is necessarily of the form \(e^{u_j}=2 \gamma G(i_0,j)\) for some
constant \(\gamma \in \R\). The normalization \(u_{i_0}=0\) implies \(\gamma= {1\over 2G(i_0,i_0)}\).
Hence the unique solution of the system~\eqref{system} is given by \(u_j=u(i_0,j)\) defined in Theorem~\ref{thm3}.

Consider the following map:
\begin{eqnarray}
\nonumber
\Phi\;:\;& \ddd&\rightarrow \{(u_j)_{j\in V}\in \R^V, \;\; u_{i_0}=0\}\times (\R_+\setminus\{0\})
\\
\label{def:phi}
&(\beta)&\mapsto ((u_j), \gamma),
\end{eqnarray}
where \((u_j)\) is the unique solution of the system~\eqref{system} and \(\gamma={1\over 2 G(i_0,i_0)}\).

We first prove that \(\Phi\) is a diffeomorphism. By the previous argument it is well-defined and injective. Reciprocally,
starting from \(((u_j),\gamma)\) on the right hand side, we define \((\beta_i)\) by
\begin{equation}
\label{beu}
\beta_i= \sum_{j\sim i}\frac{1}{2} W_{i,j} e^{u_j-u_i} +\indic_{i=i_0} \gamma.
\end{equation}
It is clear that with this definition, \((u_j)\) is the solution of~\eqref{system} with \((\beta_j)\). It remains to prove that
\(2\beta -P>0\): it is a consequence Theorem~(2.3)- (J30) of~\cite{berman1994nonnegative}:
\begin{prop}
\label{mmat}
  Let \(A\in Z_{n}=\{M\in M_{n}(\mathbb{R}),\ m_{i,j}\leq 0,\text{ if } i\neq j\}\). Then \(A\) is positive stable\footnote{All of its eigenvalues have positive real part.} if and only if there exists \(\xi\gg 0\)\footnote{\(\xi\gg \eta\) means for any coordinate \(i\), \(\xi_{i}>\eta_{i}\)} with \(A\xi>0\)\footnote{\(\xi>0\) means \(\xi_{i}\geq 0\) and \(\xi\neq 0\)} and
\begin{equation}
\label{eq-to-justify}
\sum_{j=1}^{k}a_{k,j}\xi_{j}>0,\ k=1,\ldots,n.
\end{equation}
\end{prop}
\noindent We will choose a bijection \(\sigma\) between \(V\) and \(\{1,\ldots,|V|\}\), and apply Proposition \ref{mmat} with 
\[A=((2\beta-P)_{\sigma^{-1}(i),\sigma^{-1}(j)})_{1\le i,j \le |V|},\,\,\,\xi=(e^{u_{\sigma^{-1}(i)}})_{1\le i\le |V|}.\]

 Obviously, \(\xi\gg 0\), and \(A\xi >0\) follows from \((2\be-P)e^{u.}=\delta_{i_0}/(2G(i_0,i_0))\). Now fix any spanning tree \(\Tc\) of the graph and its corresponding distance \(d\) on \(V\) throughout  the tree. Choose \(\s\) so that \(\s(i_0)=|V|\), and \(\s(i)<\s(j)\) if \(d(i_0,i)>d(i_0,j)\): this implies that, for all \(k<|V|\), there exists \(l>k\) such that \(W_{\s^{-1}(k),\s^{-1}(l)}>0\) and therefore that \eqref{eq-to-justify} holds. We conclude that \(2\beta-P>0\).
\subsection{Proof of Theorem \ref{thm3}}
We give two proofs.

{\it First proof: }We make the change of variable given by \(\Phi^{-1}\), in \eqref{def:phi} and we now prove that if \(\beta\) has distribution
\(\nu^{W,\phi^2}\), then \((u,\gamma)=\Phi^{-1}(\beta)\) has distribution \(\qqq^{W,\phi}_{i_0}\otimes \Gamma(\demi, \frac{1}{\phi_{i_0}^2})\).

Let \(J\) be the Jacobian matrix of
  \(\Phi^{-1}\) (i.e. \(J_{i,j}=\frac{\partial \beta_i}{\partial u_j},
  j\neq i_0\, J_{i,i_0}=\frac{\partial \beta_i}{\partial \gamma}\)), then \[J_{i,j}=
  \begin{cases}
    \delta_{i,i_0} & \text{ if } j=i_0, \\
    \frac{1}{2}W_{i,j}e^{u_j-u_i} & \text{ if } i\neq j, \ j\neq i_0, \\
    -\beta_i & \text{ if }i=j\neq i_0.
  \end{cases}\]
We can factorize the \(i\)th row of \(J\) by \(e^{-2u_i}\) for
  each \(i\), then expand the resulting matrix according to the
  \(i_0\)th column, and we find that
  \[|J|=\frac{1}{2^{|V|-1}}e^{-2\sum_i u_i}D(W,u)\] 
  On the other
  hand, by \eqref{beu} we deduce 
  \[|2\beta-P|=2\gamma e^{-2\sum_i u_i}D(W,u).\]
Let \(\psi\) be a positive test function. We have 
  \begin{align*}
 &  \int \psi(u,\gamma) \nu^{W,\phi^2}(d\beta)
\\
 =&
   \int \psi(u,\gamma)
   2^{|V|/2}\frac{\prod_{i}\phi_i}{\pi^{|V|/2}}\frac{\exp(-\sum_i\beta_i\phi_i^2+\sum_{\{i,j\}\in E}W_{i,j}\phi_i\phi_j)}{\sqrt{2\gamma e^{-2\sum_i u_i}D(W,u)}} \frac{1}{2^{|V|-1}}e^{-2\sum_i u_i}D(W,u)dud\gamma\\
    =&
   \int \psi(u,\gamma)\frac{\prod_{i}\phi_i}{(2\pi)^{(|V|-1)/2}}e^{-\sum_i
      u(i_0,i)}e^{-\frac{1}{2}\sum_{i\sim
        j}W_{i,j}(e^{u_i-u_j}\phi_j^2+e^{u_j-u_i}\phi_i^2-2\phi_i\phi_j)}\sqrt{D(W,u)}\cdot
    \frac{e^{-\phi_{i_0}^2\gamma}}{\sqrt{\pi\gamma}}dud\gamma
    \\
=&
   \int \psi(u,\gamma) \qqq^{W,\phi}_{i_0}(du)      \frac{\phi_{i_0}e^{-\phi_{i_0}^2\gamma}}{\sqrt{\pi\gamma}}d\gamma.
  \end{align*}
This concludes the proof of Theorem~\ref{thm3} and of Corollary~\ref{coro-thm3}.

{\it Second proof: }This proof does not make use of the explicit expression of law \(\mathcal{Q}_{i_0}^{W,\phi}\) of \(U\) in \eqref{density_u}, but rather deduces its Laplace transfom  from direct computation of the probability of a path. {Note that compared to the first proof, this one uses the representation
of the VRJP as a mixture of Markov Jump Processes, cf Theorem 2 of \cite{sabot2011edge} or Theorem \ref{thm-ST} in section \ref{sec:vrjp}, and hence it uses implicitly that
the measure \(\mathcal{Q}_{i_0}^{W,\phi}\) is a probability measure.}

We will show that, if \((u,\gamma)\) has distribution \(\qqq^{W,\phi}_{i_0}\otimes \Gamma(\demi, \frac{1}{\phi_{i_0}^2})\), then 
\(\beta=\Phi(u,\gamma)\) has distribution \(\nu^{W,\phi^2}\), which clearly implies the result.

It follows by direct computation (see \cite{sabot2013ray}, proof of Theorem 3) that the probability that, at time \(t\), the VRJP \(Z\) has followed a path \(Z_0=x_0\), \(x_1\), \(\ldots\), \(Z_t=x_n\) with jump times
respectively in \([t_i,t_i+dt_i]\), \(i=1\ldots n\),  where \(t_0=0<t_1<\ldots<t_n<t=t_{n+1}\), is \(p_tdt\), where
\begin{align*}
p_t&=\exp\left(-\sum_{\{i,j\}\in E}W_{i,j}\left(\sqrt{\phi_i^2+\ell_i}\sqrt{\phi_j^2+\ell_j}-\phi_i\phi_j\right)\right)
    \prod_{i\ne i_0}\frac{\phi_{i}}{\sqrt{\phi_i^2+\ell_i}}\\
dt&=\prod_{i=1}^{n}\frac{1}{2}W_{x_{i-1}x_i}\,dt_i,
\end{align*}
with  \((\ell_i)_{i\in V}=(\ell_i(t))_{i\in V}\)  local time at time \(t\).

On the other hand, using that, conditionally on \(U=(U_i)_{i\in V}\) in Theorem \ref{thm-ST}, \(Z\) is a Markov jump process with jump rate \(W_{ij}e^{U_j-U_i}/2\) from \(i\) to \(j\), this probability of a path is also \(q_tdt\), where
\[q_t= \int e^{-\sum_{i\in V} \tbe_i\ell_i}\mathcal{Q}_{i_0}^{W,\phi}(du)\]
and \(\tbe\) is defined in \eqref{def:tbe}.

Let \(\Gamma=\Gamma(\demi, \frac{1}{\phi_{i_0}^2})\). By identification of \(p_t\) and \(q_t\) we deduce that 
\begin{align*}
&\int e^{-\sum_{i\in V} \be_i\ell_i}\mathcal{Q}_{i_0}^{W,\phi}(du)\Gamma(d\g)
=\int e^{-\sum_{i\in V} \tbe_i\ell_i}\mathcal{Q}_{i_0}^{W,\phi}(du)\int e^{-\ell_{i_0}\g}\Gamma(d\g)\\
&=\exp\left(-\sum_{\{i,j\}\in E}W_{i,j}\left(\sqrt{\phi_i^2+\ell_i}\sqrt{\phi_j^2+\ell_j}-\phi_i\phi_j\right)\right)
   \left( \prod_{i\ne i_0}\frac{\phi_{i}}{\sqrt{\phi_i^2+\ell_i}}\right)\frac{1}{\sqrt{1+\ell_{i_0}/\phi_{i_0}^2}},
\end{align*}
which shows that the distribution \(\qqq^{W,\phi}_{i_0}\otimes \Gamma(\demi, \frac{1}{\phi_{i_0}^2})\) has the same Laplace transform as \(\nu^{W,\phi^2}\) in Proposition \ref{prop-main}.

\appendix
\section{Proof of Lemma~\ref{lem1}}
\label{pfOfLemma1}
\begin{proof}
We perform successive Gauss elimination on \(M\) to make it upper triangular. Denote by \(l_{1},\ldots,l_{n}\) the \(n\) rows of any \(n\times n\) matrix. Firstly, let 
\[M^{(1)}=M=
\begin{pmatrix}
  x^{(1)}_{1}& -H^{(1)}_{1,2} & \cdots & -H^{(1)}_{1,n} \\
  -H^{(1)}_{1,2} & x^{(1)}_{2} & \cdots & -H^{(1)}_{2,n}\\
   \cdots & \cdots & \cdots & \cdots \\
   -H^{(1)}_{1,n} & -H^{(1)}_{n,2} & \cdots & x^{(1)}_{n}
\end{pmatrix}
\]
where we set, for any \(1\le i,j\le n\), \(x^{(1)}_{i}=2\beta_{i}\) and  \(H_{i,j}^{(1)}=W_{i,j}\).

We define a sequence of matrices \(M^{(k)}\) recursively, such that 
\[M^{(k)}=
\begin{pmatrix}
  x_{1}^{(1)}& -H_{1,2}^{(1)}&\cdots &\cdots&\cdots&\cdots&\cdots&-H_{1,n}^{(1)}\\
  0&x_2^{(2)} & -H^{(2)}_{2,3} &&&&& -H^{(2)}_{2,n}\\
  \vdots &0&\ddots&\ddots&&&&\vdots\\
 \vdots &&\ddots&x_{k-1}^{(k-1)}&-H_{k-1,k}^{(k-1)}&\cdots&\cdots&-H_{k-1,n}^{(k-1)}\\
  \vdots&&&0&x_{k}^{(k)}&-H_{k,k+1}^{(k)}&\cdots &-H_{k,n}^{(k)}\\
  \vdots&&&\vdots&-H_{k,k+1}^{(k)}&\ddots&\ddots&\vdots\\
 \vdots&&&\vdots&\vdots&&\ddots&-H_{n-1,n}^{(k)}\\
  0&0&\cdots&0&-H_{k,n}^{(k)}&\cdots&-H_{n-1,n}^{(k)}&-x_{n}^{(k)}
\end{pmatrix},
\]
 by the following rule: \(M^{(k+1)}\) is constructed from \(M^{(k)}\)  by addition of columns 
\(l_{k+1}\leftarrow l_{k+1}+\frac{H^{(k)}_{k,k+1}}{x^{(k)}_{k}}l_{k},\ldots, l_{n}\leftarrow l_{n}+\frac{H^{(k)}_{k,n}}{x^{(k)}_{k}}l_{k}\) in \(M^{(k)}\). In other words,
\[
T_{k}M^{(k)}=M^{(k+1)}, \;\; \hbox{ where }\;\; 
[T_{k}]_{i,j}=
\begin{cases}
  1 & i=j\\
  \frac{H^{(k)}_{k,i}}{x^{(k)}_{k}} & i>j=k\\
  0 & \text{otherwise}
\end{cases}
\]

It is easy to check that \((x_i^{(k)})_{i\ge k}\), \((H^{(k)}_{i,j})_{i,j\ge k}\) satisfy the following induction rule:
\[
\begin{cases}
H_{i,j}^{(k+1)}=H_{i,j}^{(k)}+\frac{H_{k,i}^{(k)}H_{k,j}^{(k)}}{x_{k}^{(k)}}, & i,j\ge k+1,
\\
x_{i}^{(k+1)}=x_{i}^{(k)}-\frac{(H_{k,i}^{(k)})^{2}}{x_{k}^{(k)}},& i\ge k+1.
\end{cases}
\]

At step \(n\), we have
\[T_{n-1}\cdots T_1 M=M^{(n)}=
\begin{pmatrix}
    x^{(1)}_1& -H^{(1)}_{1,2}&\cdots&-H^{(1)}_{1,n}\\
    0&x^{(2)}_2&\cdots&-H^{(2)}_{2,n}\\
   \vdots &\ddots&\ddots&-H^{(n-1)}_{n-1,n}\\
    0&\cdots&0&x^{(n)}_n
  \end{pmatrix}\]
Hence, it gives the LU-decomposition of \(M\) where \(L^{-1}=T=T_{n-1}T_{n-2}\cdots T_{1}\) and \(U=M^{(n)}\). It is easy to check that 
\[
\begin{cases}
  x_{i}=x_{i}^{(i)} & i=1,\ldots,n\\
  H_{i,j}= H_{i,j}^{(i)} & i<j
\end{cases}
\]
satisfy the recursion in the statement, and that \(x_i=M(1,\ldots,i|1,\ldots,i)/M(1,\ldots,i-1|1,\ldots,i-1)\).

\end{proof}
\section{Time rescaling}
\label{timechange}
Let \(Y_{s}\) be the VRJP with conductances \((W)\) and initial local time \((\phi_i)_{i\in V}\) defined in Section~\ref{sec:vrjp}.
Recall that \(L_{i}(t)=\phi_{i}+\int_{0}^{t}\mathds{1}_{Y_{s}=i}ds\).
Consider the increasing functional \(A(s)=\sum_{i}(\frac{L_{i}(s)}{\phi_{i}}-1)\), and  the time-changed 
process \(\tilde{Y}_{\tilde s}=Y_{A^{-1}(\tilde s)}\). Let 
\[
\tilde L_i(\tilde s)= 1+ \int_{0}^{t}\mathds{1}_{\{\tilde Y_{\tilde s}=i\}}d\tilde s.
\]
We always denote by \(\tilde s\) the time scale of \(\tilde Y\), we can write 
\[
\tilde s=A(s),\;\;  d\tilde s= {d s\over \phi_{Y_s}},\;\;  L_i(\tilde s)= {1\over \phi_i} L_i(s).
\]
Obviously, \(\tilde Y\)
is a VRJP with edge weight \(W_{i,j}\phi_{i}\phi_{j}\) and initial local local time 1 : that is, 
conditionally on \(\mathcal{F}^{\tilde Y}_{\tilde s}\), \(\tilde{Y}\) jumps from \(i\) to \(j\) at rate
\[W_{i,j}\phi_{i}\phi_{j}\tilde L_{j}(\tilde s).\]
Note for simplicity
 \[
 W^\phi_{i,j}= W_{i,j}\phi_i\phi_j.
 \]
We can apply~\cite{sabot2011edge} Theorem 2 to \(\tilde Y\). Let 
\[\tilde D(\tilde s)=\sum_{i}\tilde L_i(\tilde s)^2-1,\] 
and set \(\tilde Z_{\tilde t}=\tilde{Y}_{\tilde D^{-1}(\tilde t)}\), with local time \(\tilde \ell_{i}(\tilde t)=\int_{0}^{\tilde t}\mathds{1}_{\tilde X_{u}=i}du\). 
By proposition 1 of~\cite{sabot2011edge} translated in time scale \(L\) (cf relation (2.1) of~\cite{sabot2011edge}), we have that
\(\log \tilde L_i(\tilde s)-{1\over N}\sum_{j\in V}  \log \tilde L_j(\tilde s)\) converges a.s. when \(\tilde s\to \infty\) to a random vector with distribution
given by (3.1) of theorem 1 of~\cite{sabot2011edge}, where the  weights \((W_{i,j})\) are replaced by
 \((W^\phi_{i,j})\). Changing to variables \(u_i\to u_i-u_{i_0}\), we deduce 
\[
\lim_{\tilde s\to \infty} \log \tilde L_i(\tilde s)-\log \tilde L_{i_0}(\tilde s)=U_i
\]
exists and has distribution 
\begin{equation*}
\mathcal{Q}^{W^\phi}_{i_0}(du)=
{1 \over \sqrt{2\pi}^{N-1}} e^{-\sum_{j\in V} u_j} 
e^{-\frac{1}{2}\sum_{i\sim j}W^\phi_{i,j}(\cosh(u_i-u_j)-1)}
\sqrt{D(W^\phi,u)}\; du,
\end{equation*}
and that \(\tilde Z\) is a mixture of Markov Jump Process with jumping rates \(\demi W^\phi_{i,j} e^{U_j-U_i}\). We now come back to
\((Z_t)\). Recall that \(Z_t=Y_{D^{-1}(t)}\), where \(D(t)\) is defined in (\ref{Dt}).
>From this we have
\[
\tilde t=\tilde D(A(D^{-1}(t))),
\]
and 
\[
d\tilde t={1\over \phi_{\tilde Y_{\tilde s}}} {\tilde L_{\tilde Y_{\tilde s}}(\tilde s)\over L_{Y_s}(s)}dt= {1\over  \phi^2_{Z_{t}}}dt.
\]
This implies that \((Z_t)\) is a mixture of Markov Jump  processes with jumping rates 
\(\demi W_{i,j} e^{U_j+\log\phi_j-U_i-\log\phi_i}\). 
By simple change of variables, \(U_i+\log \phi_i-\log\phi_{i_0}\) has distribution
\begin{equation*}
\mathcal{Q}^{W,\phi}_{i_0}(du)={\prod_{j\neq i_0} \phi_j \over \sqrt{2\pi}^{N-1}} e^{-\sum_{j\in V} u_j} e^{-\frac{1}{2}\sum_{i\sim
        j}W_{i,j}(e^{u_i-u_j}\phi_j^2+e^{u_j-u_i}\phi_i^2-2\phi_i\phi_j)
        } \sqrt{D(W,u)}\; du.
\end{equation*}

\label{Appendix}
\noindent{{\it Acknowledgment : }
The authors are very grateful to G\'erard Letac for a useful remark at an early stage of this work.
They are also grateful to Persi Diaconis and G\'erard Letac for interesting discussions about the measure that appears in Theorem 
\ref{thm_main}.
}

\bibliography{bib}
\bibliographystyle{plain}
\end{document}